\documentclass[12pt]{amsart}
\usepackage{amsmath,amsthm,amsfonts,amscd,amssymb,eucal,latexsym,mathrsfs, appendix, yhmath, setspace}
\usepackage{color,colortbl}
\usepackage[numbers,sort&compress]{natbib}
\usepackage[all,cmtip]{xy}
\usepackage{enumerate}

\newtheorem{theorem}{Theorem}[section]
\newtheorem{corollary}[theorem]{Corollary}
\newtheorem{lemma}[theorem]{Lemma}
\newtheorem{proposition}[theorem]{Proposition}

\theoremstyle{definition}
\newtheorem{definition}[theorem]{Definition}
\newtheorem{remark}[theorem]{Remark}

\DeclareMathOperator{\Ima}{im \ }
 \newcommand{\Ker}{{\rm ker \ }}

\newcommand{\hso}{ h_{\Sigma, \omega}}

 \newcommand{\Map}{{\rm Map}}

\newcommand{\rL}{{\rm L}}
\newcommand{\Sym}{{\rm Sym}}

  \newcommand{\cF}{{\mathcal F}}
  \newcommand{\cJ}{{\mathcal J}}
  
 \newcommand{\cM}{{\mathcal M}}
 \newcommand{\cN}{{\mathcal N}}

    \newcommand {\cW}{{\mathcal W}}

  \newcommand{\bF}{{\mathbb F}}
 \newcommand{\bN}{{\mathbb N}}

 \newcommand{\bR}{{\mathbb R}}
 
 \newcommand{\bZ}{{\mathbb Z}}
  
 \newcommand{\ZG}{{\mathbb Z \Gamma}}

 \newcommand{\hcM}{\widehat{\mathcal M}}

 \newcommand{\sA}{{\mathscr A}}
\newcommand{\sB}{{\mathscr B}}

\newcommand{\sF}{{\mathscr F}}

\newcommand{\sM}{{\mathscr M}}
\newcommand{\sN}{{\mathscr N}}

\allowdisplaybreaks

\begin{document}

\title{Entropy on modules over the group ring of a sofic group }

\author{Bingbing Liang}

\address{\hskip-\parindent
B.L., Max Planck Institute for Mathematics, Vivatsgasse 7, 53111, Bonn, Germany }
\email{bliang@mpim-bonn.mpg.de}
\email{bbliang2008@163.com}


\date{September 5, 2017}

\begin{abstract}
We partially generalize Peters' formula \cite{Peters79} on modules over the group ring $\bF \Gamma$ for a given finite field $\bF$ and a sofic group $\Gamma$. It is also discussed that how the values of entropy are related to the zero divisor conjecture.
\end{abstract}

\maketitle

\section{Introduction} 

Let $\Gamma$ be a countable discrete group. An {\it algebraic action} of $\Gamma$ is the induced (continuous) action $\Gamma \curvearrowright \hcM$ from some (left) $\ZG$-module $\cM$. Here $\widehat{\cM}$ is the Pontryagin dual of the discrete abelian group $\cM$. One of main concerns is about the interplay between the dynamical information of $\Gamma \curvearrowright \hcM$ and the algebraic information of $\cM$. Weiss considered the case that $\cM$ is a torsion abelian group \cite{Weiss74}. For general $\ZG$-module $\cM$, Peters showed that the topological entropy of $\Gamma \curvearrowright \hcM$ coincides with the algebraic entropy of $\cM$  when $\Gamma$ is amenable \cite{Peters79}.  This general result does not only establish the beautiful connection but also provides the great flexibility to study the entropy of algebraic actions \cite{CL15, LT14}.

Towards more general group actions, Bowen and Kerr-Li developed an entropy theory when $\Gamma$ can be approximated by finite groups \cite{Bowen10, KL11}. The groups  admitting such approximations are the so-called {\it sofic groups}, which in particular includes amenable groups and residually groups \cite{Gromov99S, Weiss99}.

On the other hand, given a unital ring $R$ and a sofic group $\Gamma$, in \cite[Definition 3.1]{LL15A}, started with a length function $\rL$ on left $R$-modules (\cite[Definition 2.1]{LL15A}), Li and the author introduced a mean length function  on  $R\Gamma$-modules $\cM$ which are {\it locally $L$-finite} in the sense that each finitely generated $R$-submodule of $\cM$ has finite $\rL$-length. In particular, for $R=\bZ$ and $\rL(\cdot)=\log|\cdot|$, the corresponding mean length function can be treated as an algebraic version of entropy, which we may call {\it sofic algebraic entropy}. Then a natural question to ask is whether one can generalize Peters' formula to the setting of sofic group actions.

In this paper, we give a partial generalization of Peters' formula to the case of sofic group actions. The main inputs of the proof are an approximation formula of topological entropy given in \cite[Lemma 4.8]{GS15A} and some techniques appeared in \cite{LL15A}. Given  a finite field $\bF$ any $\bF \Gamma$-module can be treated as a $\ZG$-module. Fix a sofic approximation $\Sigma$ for $\Gamma$ and a free ultrafilter $\omega$ over $\bN$. For any left $\bF \Gamma$-modules $\cM_1 \subseteq \cM_2$, we have the algebraic entropy of $\cM_1$ relative to $\cM_2$, denoted by $\hso(\cM_1|\cM_2)$ (Definition \ref{sofic algebraic entropy}). For any $\bF \Gamma$-module $\cM$, the sofic algebraic entropy $\hso(\cM)$ of $\cM$ is then defined as $\hso(\cM|\cM)$. We remark that the main motivation to introduce this relative version of invariants and technical consideration of ultrafilters is that one can establish a modified addition formula under this approach \cite[Theorem 1.1]{LL15A}.

Correspondingly, we have the relative version of topological entropy in dynamical systems. Let $\Gamma$ act on compact metrizable spaces $X$ and $Y$ continuously. For any factor map from $X$ to $Y$, we have the topological entropy $\hso(Y|X)$ of $\Gamma \curvearrowright Y$ relative to $\Gamma \curvearrowright X$ \cite[Definition 9.3]{LL15A}. When $X=Y$ and the factor map is the identity map, the relative topological entropy coincide with the original definition of sofic topological entropy, which we denote by $\hso(X)$. Similar approaches also independently appears in the works of other experts. Hayes gave a formula for the relative version of sofic measure entropy in terms of a given compact model in \cite{Hayes15}. A similar notion for Rokhlin entropy, called outer Rokhlin entropy, was developed by Seward in  \cite{SewardII}.

\begin{theorem} \label{main theorem}
Let $\Gamma$ be a sofic group and $\bF$ be a finite field. Then for any finitely generated $\bF \Gamma$-module $\cM$, we have
$$h_{\Sigma, \omega}(\hcM)=h_{\Sigma, \omega}(\cM).$$
Let $\cM_1 \subseteq \cM_2$ be countable $\bF \Gamma$-modules. Then $h_{\Sigma, \omega}(\hcM_1|\hcM_2)\leq h_{\Sigma, \omega}(\cM_1|\cM_2).$
\end{theorem}

This paper is organized as follows. We recall some background knowledge in Section 2. The algebraic entropy is introduced  in Section 3 and some basic properties are discussed there. In particular, we show that the values of algebraic entropies is related to the zero divisor conjecture. We prove the main result in Section 4.

Throughout this paper, $\Gamma$ will be a countable discrete group. For any set $S$, we denote by $\cF(S)$ the set of all nonempty finite subsets of $S$. All modules are assumed to be left modules unless specified. For any $d \in \bN$, we write $[d]$ for the set $\{1, \cdots, d\}$ and $\Sym(d)$ for the permutation group of $[d]$.

\noindent{\it Acknowledgements.}
We thank the helpful comments from Yongle Jiang and Hanfeng Li.

\section{Preliminaries}

\subsection{Group rings and algebraic actions}
Let $R$ be a unital ring. The {\it group ring of $\Gamma$ with coefficients in $R$}, denoted by $R \Gamma$, consists of all finitely supported functions: $f: \Gamma \to R$. Conventionally, we shall write $f$ as $\sum_{s \in \Gamma} f_s s$, where $f_s \in R$ for all $s \in \Gamma$ and $f_s=0$ for all except finitely many $s \in \Gamma$. 
The algebraic operations on $R \Gamma$ are defined by
$$ \sum_{s\in \Gamma}f_ss+\sum_{s\in \Gamma}g_ss=\sum_{s\in \Gamma}(f_s+g_s)s, \mbox{ and } \big(\sum_{s\in \Gamma}f_s s\big)\big(\sum_{t\in \Gamma}g_tt\big)=\sum_{s, t\in \Gamma}f_sg_t(st).$$
We similarly have the product if one of $f$ and $g$ sits in $R^\Gamma$.

For any countable $\ZG$-module $\cM$, treated as a discrete abelian group, its Pontryagin dual $\hcM$ consisting of all continuous group homomorphisms $\cM \to \bR/\bZ$,  coincides with ${\rm Hom}_{\bZ}(\cM, \bR/\bZ)$. By Pontryagin duality, $\hcM$ is a compact metrizable space under compact-open topology. Furthermore, the $\ZG$-module structure of $\cM$ naturally induces an adjoint action $\Gamma \curvearrowright \hcM$ by continuous automorphisms. To be precise,
$$\langle s\chi, u\rangle: =\langle \chi, s^{-1}x \rangle$$
for all $\chi \in \hcM, u \in \cM$, and $s \in \Gamma$.
Conversely, by Pontryagin duality, each action of $\Gamma$ on a compact metrizable abelian group arise this way and thus we call such a dynamical system an  {\it algebraic action} \cite{Schmidt95}.

\subsection{Amenable and sofic groups}
The group $\Gamma$ is called {\it amenable} if for any $K\in \cF(\Gamma)$ and any $\delta>0$ there is an $F\in \cF(\Gamma)$ with $|KF\setminus F|<\delta |F|$.

A sequence of maps $\Sigma=\{\sigma_i: \Gamma \to {\rm \Sym} (d_i)\}_{i \in \bN}$ is called a {\it sofic approximation} for $\Gamma$  if it satisfies:
\begin{enumerate}
\item $\lim_{i\to \infty}|\{v\in [d_i]: \sigma_{i,s}\sigma_{i,t}(v)=\sigma_{i, st}(v)\}|/d_i=1$ for all $s, t\in \Gamma$,

\item $\lim_{i\to \infty}|\{v\in [d_i]: \sigma_{i, s}(v)\neq \sigma_{i,t}(v)\}|/d_i=1$ for all distinct $s, t\in \Gamma$,

\item $\lim_{i\to \infty} d_i=+\infty$.
\end{enumerate}
The group $\Gamma$ is called a {\it sofic group} if it admits a sofic approximation.

Any amenable group is sofic since one can use a sequence of asymptotically-invariant subsets of the amenable group, i.e. {\it F{\o}lner sequence},  to construct a sofic approximation. Residually finite groups are also sofic since a sequence of exhausting finite-index subgroups naturally induces a sofic approximation in which each approximating map is actually a group homomorphism. We refer the reader to \cite{CL15B, CC10B} for more information on sofic groups.

Throughout the rest of this paper, $\Gamma$ will be a countable sofic group,  $\Sigma$ will be a sofic approximation for $\Gamma$, and $\omega$ will be a free ultrafilter over $\bN$.

\subsection{Sofic topological entropy} 
For any continuous pseudometric $\rho_X$ on a compact space $X$ and any $\varepsilon >0$, a subset $Z$ of $X$ is called {\it ($\rho_X, \varepsilon$)-separated} if $\rho_X(z_1, z_2) > \varepsilon$ for all distinct $z_1, z_2 \in Z$. Set $N_\varepsilon(X, \rho_X):=\max_Z|Z|$ for $Z$ ranging over all $(\rho_X, \varepsilon)$-separated subsets of $X$.

Let $\Gamma$ act continuously on a compact metrizable space $X$.

\begin{definition}
Let $\rho_X$ be a continuous pseudometric on $X$. For any $d \in \bN$, define continuous pseudometrics $\rho_{X, 2}$ and $\rho_{X, \infty}$ on $X^d$ by
$$\rho_{X, 2}(\varphi, \psi)=\left(\frac{1}{d}\sum_{v \in [d]} \rho_X(\varphi_v, \psi_v)^2\right)^{1/2}, {\rm and \ } \rho_{X, \infty}(\varphi, \psi)=\max_{v \in [d]} \rho_X(\varphi_v, \psi_v).$$
Let $\sigma$ be a map from $\Gamma$ to $\Sym(d)$, $F \in \cF(\Gamma)$, and $\delta > 0$. The set of approximately equivariant maps $\Map(\rho, F, \delta, \sigma)$ is defined to be the set of all maps $\varphi: [d] \to X$ such that $\rho_{X, 2}(s\varphi, \varphi \circ \sigma(s)) \leq \delta$.

A continuous pseudometric $\rho_X$ on $X$ is called {\it dynamically generating} if 
$\rho_X$ can distinguish all distinct elements of  $X$ under the action of $\Gamma$, i.e.  for all distinct $x, x' \in X$, there exists some $s \in \Gamma$ such that $\rho_X(sx, sx')> 0$.
\end{definition}

Now let $\Gamma$ act on another compact metrizable space $Y$ and $\pi: X \to Y$ be a factor map. Denote by $\Map(\pi, \rho, F, \delta, \sigma)$ the set of all $\pi \circ \varphi$ for $\varphi$ ranging in $\Map(\rho, F, \delta, \sigma)$.

\begin{definition}
Let $\rho_X$ and $\rho_Y$ be two dynamically generating continuous pseudometrics of $X$ and $Y$ respectively. Let $F \in \cF(\Gamma)$ and $\delta, \varepsilon > 0$. Define
$$h^\varepsilon_{\Sigma, \omega, \infty}(\rho_Y, F, \delta|\rho_X)=\lim_{i\to \omega}\frac{1}{d_i}\log N_\varepsilon(\Map(\pi, \rho_X, F, \delta, \sigma_i), \rho_{Y, \infty}).$$
If the set of $i\in \bN$ with $\Map(\rho_X, F, \delta, \sigma_i)=\emptyset$ is in $\omega$, we set $h^\varepsilon_{\Sigma, \omega, \infty}(\rho_Y, F, \delta|\rho_X)=-\infty$. We define the {\it sofic topological entropy of $\Gamma \curvearrowright Y$ relative to $\Gamma \curvearrowright X$} as
$$\hso(Y|X): =\sup_{\varepsilon > 0} \inf_{F \in \cF(\Gamma)} \inf_{\delta >0}h^\varepsilon_{\Sigma, \omega, \infty}(\rho_Y, F, \delta|\rho_X). $$
The {\it sofic topological entropy of $\Gamma \curvearrowright X$} is defined as 
$\hso(X):=\hso(X|X)$
for $\pi: X \to X$ being the identity map.
\end{definition}

\begin{remark}
It was shown that the definition of the relative sofic topological entropy does not depend on the choices of the dynamically generating continuous pseudometrics $\rho_X$ and $\rho_Y$ \cite[Lemma 8.2, Lemma 9.5]{LL15A}. One has $\hso(Y|X) \leq \min (\hso(X), \hso(Y))$. 
\end{remark}
\section{Algebraic entropy}

When we specialize the length function as $\log|\cdot|$ on $\bZ$-mdoules, the corresponding sofic mean length function in \cite[Definition 3.1]{LL15A} brings us the notion for the algebraic analogue of sofic topological entropy. For reader's convenience, we recall the complete definition.

For any (abelian) group $G$, denote by $\sF(G)$ the collection of finitely generated subgroups of $G$. Let $\cM$ be a $\ZG$-module and $\sA, \sB \in \sF(\cM)$. For  $F \in \cF(\Gamma)$ and a map $\sigma: \Gamma \to \Sym(d)$ for some $d \in \bN$, denote by $\sM(\sB, F, \sigma)$ the abelian subgroup of $\cM^d \cong \bZ^d \otimes_\bZ \cM$ generated by the elements $\delta_v\otimes b-\delta_{\sigma(s)v}\otimes sb$ for all $v \in [d], b \in \sB$, and $s \in F$, where $\delta_v$ denotes the element of $\bZ^d$ taking value 1 at the coordinate $v$ and $0$ everywhere else.  Put $\sM(\sA, \sB, F, \sigma)$ as the image of $\sA^d$ in $\cM^d/\sM(\sB, F, \sigma)$ under the quotient map $\cM^d \to \cM^d/\sM(\sB, F, \sigma)$. We remark that any $\ZG$-module $\cM$ is locally $\log|\cdot|$-finite if and only if $\cM$ is torsion as an abelian group.

\begin{definition} \label{sofic algebraic entropy}
Let $\cM_1 \subseteq \cM_2$ be $\ZG$-modules such that $\cM_1$ is torsion. For any $\sA \in \sF(\cM_1), \sB \in \sF(\cM_2)$ and $F \in \cF(\Gamma)$, set
$$\hso(\sA|\sB, F):=\lim_{i \to \omega} \frac{\log|\sM(\sA, \sB, F, \sigma_i)|}{d_i}.$$
We define the {\it algebraic entropy of $\cM_1$ relative to $\cM_2$} as
$$\hso(\cM_1|\cM_2):=\sup_{\sA \in \sF(\cM_1)} \inf_{F \in \cF(\Gamma)} \inf_{\sB \in \sF(\cM_2)}\hso(\sA|\sB, F).$$
The {\it sofic algebraic entropy} of $\cM_1$ is then defined as 
$$\hso(\cM_1):=\hso(\cM_1|\cM_1).$$
\end{definition}

Apply \cite[Theorem 1.1]{LL15A} to the length function $\log|\cdot|$ on $\bZ$-modules, we have the modified addition formula for sofic algebraic entropy.

\begin{theorem} \label{addition}
Let $\cM_1 \subseteq \cM_2$ be $\ZG$-modules such that $\cM_2$ is torsion as an abelian group. Then
$$\hso(\cM_2)=\hso(\cM_1|\cM_2)+\hso(\cM_2/\cM_1).$$
\end{theorem}

By the similar argument as in \cite[Section 3]{LL15A}, we have the following proposition, which collects basic properties of the sofic algebraic entropy. 

\begin{proposition} \label{property}
Let $\cM_1 \subseteq \cM_2$ be $\ZG$-modules such that $\cM_2$ is torsion as an abelian groups. Then the following are true.
\begin{enumerate}
\item If $\sN_1 \subseteq \sN_2$ are two torsion abelian groups, then $\hso(\ZG \otimes_\bZ \sN_1|\ZG \otimes_\bN \sN_2)=\log|\sN_1|$;
    \item If  $\{\cM_j'\}_{j \in \cJ}$ is an increasing net of $\bZ \Gamma$-submodules of $\cM_1$, then $\hso(\cM_j'|\cM_2) \nearrow \hso(\bigcup_{j \in \cJ} \cM_j'|\cM_2)$. Morerover, if $\hso(\cM_2) < \infty$, then $\hso(\cM_2/\cM_j') \searrow \hso(\cM_2/\bigcup_{j \in \cJ} \cM_j')$;
    \item If $\cM_1$ is a finitely generated $\ZG$-modules, and $\{\cM_j'\}_{j \in \bN}$ is an increasing sequence of $\ZG$-submodules of $\cM_2$ containing $\cM_1$ with union $\cM_2$, then $\hso(\cM_1|\cM_j') \searrow \hso(\cM_1|\cM_2)$.
\end{enumerate}

\end{proposition}

Using the similar argument in the proof of \cite[Proposition 3.5 (2)]{LL15A}, analogues to \cite[Lemma 4.1]{CT15}, we have
\begin{lemma} \label{strict positive}
Let $R$ be a finite abelian group. Then for any $\ZG$-module $\cM \subseteq (R\Gamma)^k$ for some $k \in \bN$, we have $\cM=0$ if and only if $\hso(\cM|(R\Gamma)^k)=0$. In fact, for any $f \in (\ZG)^k$,
$$\hso(R\Gamma f|(R\Gamma)^k) \geq \frac{\log |\{rf: r \in R\}|}{(2|{\rm supp} (f)|+1)^2}.$$

\end{lemma}

\begin{proposition}
Let $R$ be a finite abelian group. Assume for any nonzero $f$ of $R \Gamma$,  we have $h_{\Sigma}(R \Gamma/R \Gamma f)=0$. Then $R \Gamma$ contains no nontrivial zero divisors.
\end{proposition}

\begin{proof}
  By Theorem \ref{addition} we have $0=\hso(R\Gamma/R \Gamma f)=\hso(R \Gamma)-\hso(R\Gamma f|R \Gamma)$. Since $\hso(R\Gamma f|R\Gamma)\leq \hso(R\Gamma f) \leq \hso(R\Gamma)$, we have $\hso(R\Gamma f)=\hso(R\Gamma)$.
  
  On the other hand, let $R(f): R\Gamma \to R\Gamma$ be the $R\Gamma$-module homomorphism sending $x$ to $xf$. Then we have $\hso(R \Gamma f)=\hso(R \Gamma)-\hso(\Ker R(f)|R \Gamma)$. It forces that $\hso(\Ker R(f)|R \Gamma)=0$. By Lemma \ref{strict positive}, $\Ker R(f)=0$. That means $R \Gamma$ has no nontrivial zero divisor.
\end{proof}

We remark the similar approaches also appear in \cite[Lemma 4.4]{CT15}, \cite[Theorem 7.14]{Virili14A}, and \cite[Theorem 4.5]{LL15A}.
\begin{definition}
When $\Gamma$ is amenable,  for any $\ZG$-module $\cM$ such that $\cM$ is torsion as an abelian group, the {\it algebraic entropy} of $\cM$ is defined as
$$h(\cM)=\sup_{\sA \in \sF(\cM)} \inf_{F \in \cF(\Gamma)} \frac{\log|\sum_{s \in F} s^{-1}\sA|}{|F|}.$$

\end{definition}

From \cite[Theorem 5.1]{LL15A}, we know the sofic  algebraic entropy is a generalization of algebraic entropy.

\begin{theorem} \label{generalization}
Let $\Gamma$ be an amenable group. For any $\ZG$-modules $\cM_1 \subseteq \cM_2$ such that $\cM_2$ is torsion as abelian group, then $\hso(\cM_1|\cM_2)=h(\cM_1)$.
\end{theorem}

\begin{remark}
Let $R$ be a finite abelian group and $f \in R H \subseteq R \Gamma$ be nonzero for some amenable subgroup $H$ of $\Gamma$. If $R H$ has no nontrivial zero divisor, we have $R H$ is isomorphic to $R H f$ as $R H$-modules. By \cite[Corollary 6.2]{LL15A} and Theorem \ref{generalization}, we have
$$\hso(R \Gamma /R \Gamma f)=h_{\Sigma|_H, \omega}(R H/R Hf)=h(R H/R Hf)=h(R H)-h(R Hf)=0.$$
\end{remark}

By Bartholdi's characterization of nonamenability \cite[Theorem 1.1]{BK16}, for any field $\bF$ there exists an injective $\bF \Gamma$-module homomorphism $\varphi: (\bF \Gamma)^n \to (\bF \Gamma)^{n-1}$ for some  $n \in \bN$. In particular, when $\bF$ is finite, one has $\hso(\Ima \varphi|(\bF \Gamma)^{n-1})< \hso(\Ima \varphi)$. Combining Theorem \ref{generalization}, we have

\begin{corollary}
$\Gamma$ is amenable if and only if  $\hso(\cM_1|\cM_2)=\hso(\cM_1)$ holds for any $\ZG$-modules $\cM_1 \subseteq \cM_2$ such that $\cM_2$ is torsion as an abelian group.
\end{corollary}

\section{Peters' formula}

In this section, we prove Theorem \ref{main theorem}, which follows from Proposition \ref{finitely generated case} and Proposition \ref{general case}. 

Let $\bF$ be a finite field. For a finitely generated free $\bF \Gamma$-module $(\bF \Gamma)^{1\times n}$, if we identify $\widehat{(\bF \Gamma)^{1\times n}}$ with $(\bF^\Gamma)^{1\times n}=(\bF^{1\times n})^\Gamma$ naturally via the pairing
$$(\bF \Gamma)^{1\times n} \times (\bF^\Gamma)^{1\times n} \to \bF, \ (u, x) \mapsto (x u^*)_{e_\Gamma},$$
the induced natural action of $\Gamma$ on $\widehat{(\bF \Gamma)^{1\times n}}$ is then given by left shift. Here $u^*=\sum_{s \in \Gamma} u_s s^{-1}$. For a finitely presented $\bF \Gamma$-module $\cM$, write it as $\cM=(\bF \Gamma)^n/(\bF \Gamma)^mf$ for some $m, n \in \bN$ and $f \in M_{m,n}(\bF \Gamma)$, we can further identify $\hcM$ as $\Ker R(f^\ast)$ for the linear map $R(f^\ast):  (\bF^\Gamma)^{1\times n} \to (\bF^\Gamma)^{1\times m}$ sending $x$ to $xf^\ast$.

\begin{proposition} \label{finitely generated case}
Let $\bF$ be a finite field. Then for any finitely generated $\bF \Gamma$-module $\cM$, we have $\hso(\hcM)=\hso(\cM)$.
\end{proposition}

\begin{proof}
We first show the case that $\cM$ is finitely presented. Write $\cM$ as $\cM=(\bF \Gamma)^n/(\bF \Gamma)^mf$ for some $m, n \in \bN$ and $f \in M_{m,n}(\bF \Gamma)$. Given a map $\sigma: \Gamma \to \Sym(d)$ for some $d \in \bN$, $f$ induces linear maps $\sigma(f): (\bF^d)^{n\times 1} \to (\bF^d)^{m\times 1}$ and $\bar{\sigma}(f): (\bF^d)^{m\times 1} \to (\bF^d)^{n\times 1}$ which are uniquely determined by the linear maps $\sigma_s: \bF^d \to \bF^d$ and $\bar{\sigma}_s: \bF^d \to \bF^d$ respectively for every $s \in \Gamma$.  Here $(\sigma_s(x))(v)=x((\sigma_s)^{-1}v) $ and $(\bar{\sigma}_s(x))(v)=x(\sigma_s(v))$ (see \cite[Section 6]{LL15A} for more details). Apply \cite[Theorem 1.1]{LL15A} and \cite[Proposition 6.1]{LL15A} to the length function $\log|\cdot|$, we have
\begin{align*}
    \hso(\cM)&=\hso((\bF \Gamma)^n)-\hso((\bF \Gamma)^m f|(\bF \Gamma)^n)\\
    &=\hso((\bF \Gamma)^n)-\varlimsup_{i \to \infty} \frac{\log |\Ima \bar{\sigma}_i(f)|}{d_i}.
\end{align*}

Following the similar argument as in the proof of \cite[Lemma 7.4]{LL15A}, we know $\Ker \sigma(f)$ is isomorphic to the dual vector space of $(\bF^d)^{n\times 1}/\bar{\sigma}(f)((\bF^d)^{m\times 1})$. Hence,
$$ \hso(\cM)=\varlimsup_{i \to \infty} \frac{\log |\Ker \sigma_i(f)|}{d_i}.$$

From \cite[Lemma 4.8]{GS15A}, we have
$$h_\Sigma(\hcM)=h_\Sigma(\Ker R(f^\ast))=\varlimsup_{i \to \infty} \frac{\log |\Ker \bar{\sigma}_i(f^\ast)|}{d_i}.$$
For a given map $\sigma: \Gamma \to \Sym(d)$ for some $d \in \bN$, put $\cW=\{v \in [d]: \sigma(s^{-1})(v)=\sigma(s)^{-1}(v), {\rm for \ all \ } s \in {\rm supp}(f)\}$. Then it is easy to verify that 
$$\Ker \sigma(f) \subseteq \Ker (P_{\cW} \circ \sigma(f))=\Ker (P_{\cW} \circ \bar{\sigma}(f^\ast)) \supseteq \Ker \bar{\sigma}(f^\ast),$$
where $P_{\cW}$ is the map projecting onto the $\cW$ coordinates.
It follows that 
$$\log|\Ker \sigma(f)| \leq \log|\Ker (P_{\cW} \circ \bar{\sigma}(f^\ast))|\le n|\cW^c|\log|\bF|+\log|\Ker \bar{\sigma}(f^\ast)|$$
and 
$$\log|\Ker \bar{\sigma}(f^\ast)| \leq \log|\Ker (P_{\cW} \circ \sigma(f))|\le n|\cW^c|\log|\bF|+\log|\Ker \sigma(f)|.$$
By soficity, $|\cW^c|/d$ approaches to zero as $d$ gets large, it follows that
$$\hso(\hcM)=\hso(\cM)=\varlimsup_{i \to \infty} \frac{\log |\Ker \sigma_i(f)|}{d_i}.$$

Now for any finitely generated $\bF \Gamma$-module $\cM$,  write $\cM$ as $\cM=(\bF \Gamma)^n/\cup_{j \in \bN} \cN_j$ for some increasing sequence of finitely generated submodules $\{\cN_j\}_{j \in \bN}$ of $(\bF \Gamma)^n$. By \cite[Corollary 3.6]{GS15A} and \cite[Proposition 3.4(1)]{LL15A}, we have
$$\hso(\hcM)==\inf_j \hso(\widehat{(\bF \Gamma)^n/\cN_j})=\inf_j \hso((\bF \Gamma)^n/\cN_j) =\hso(\cM).$$
\end{proof}

Following the similar argument as in the proof of \cite[Lemma 10.4]{LL15A}, we also have an inequality for topological entropy.
\begin{lemma} \label{inequality}
Let $X$ be a compact metrizable abelian group carrying an action of $\Gamma$ by continuous automorphisms. Then for any closed $\Gamma$-invariant subgroup $Y$ of $X$, the induced quotient map $\pi: X \to X/Y$ satisfies
$$\hso(X) \geq \hso(X/Y|X)+\hso(Y).$$
\end{lemma}

\begin{proposition} \label{general case}
Let $\bF$ be a finite field. Then for any countable $\bF \Gamma$-modules $\cM_1 \subseteq \cM_2$, we have $\hso(\widehat{\cM_1}|\widehat{\cM_2}) \leq \hso(\cM_1|\cM_2)$.

\end{proposition}

\begin{proof}
{\bf Case 1}. $\cM_1 \subseteq \cM_2$ are finitely generated.

 Combining Lemma \ref{inequality}, Proposition \ref{finitely generated case}, and Theorem \ref{addition}, we have
$$h_\Sigma(\widehat{\cM_1}|\widehat{\cM_2})\leq \hso(\widehat{\cM_2})-\hso(\widehat{\cM_2/\cM_1})=\hso(\cM_2)-\hso(\cM_2/\cM_1)= \hso(\cM_1|\cM_2).$$

{\bf Case 2}. $\cM_1$ is finitely generated.

Write $\cM_2$ as the union of an increasing sequence of finitely generated $\bF \Gamma$-modules $\{\cN_j\}_{j \in \bN}$ of $\cM_2$ containing $\cM_1$. It is clear that by definition, for all $j \in \bN$, one has
$$\hso(\widehat{\cM_1}|\widehat{\cM_2})\leq \hso(\widehat{\cM_1}|\widehat{\cN_j}).$$
Thus  by Proposition \ref{property} and Case 1, we have
$$\hso(\widehat{\cM_1}|\widehat{\cM_2})\leq \lim_{j \to \infty} \hso(\widehat{\cM_1}|\widehat{\cN_j})\leq \lim_{j \to \infty} \hso(\cM_1|\cN_j)=\hso(\cM_1|\cM_2).$$

For general case, write $\cM_1$ as the union of an increasing sequence of finitely generated submodules $\{\cN_j\}_{j \in \bN}$ of $\cM_1$. Following the similar argument as in the proof of \cite[Lemma 9.10]{LL15A}, we have
$$\hso(\widehat{\cM_1}|\widehat{\cM_2})=\lim_{j \to \infty} \hso(\widehat{\cN_j}|\widehat{\cM_2}).$$
Thus by Proposition \ref{property} and Case 2, we have
$$\hso(\widehat{\cM_1}|\widehat{\cM_2}) \leq \lim_{j \to \infty} \hso(\cN_j|\cM_2) =\hso(\cM_1|\cM_2).$$
\end{proof}


\end{document}